\documentclass[12pt]{amsart}

\usepackage{geometry}  
\usepackage{amsmath,amsthm,amssymb}
\usepackage{enumitem,thmtools}
\usepackage{xfrac}
\usepackage{tikz, calc, forest}

\usetikzlibrary{arrows.meta,shapes}
 
\newcommand{\N}{\mathbb{N}}

\newcommand{\lp}{\left(}
\newcommand{\rp}{\right)}

\newcommand{\tx}{\text}
\newcommand{\txbf}{\textbf}

\newtheorem{theorem}{Theorem}
\newtheorem{lemma}{Lemma}
\newtheorem{cor}{Corollary}
\newtheorem{prop}{Proposition}

\newtheorem*{theorem*}{Theorem}
\newtheorem*{cor*}{Corollary}
\newtheorem*{lemma*}{Lemma}
\newtheorem*{example*}{Example}

\theoremstyle{definition}
\newtheorem{definition}{Definition}
\newtheorem*{def*}{Definition}
\newtheorem{remark}{Remark}
\newtheorem*{remark*}{Remark}

\begin{document}
 
 
\title{Parking Functions on Directed Graphs and Some Directed Trees}

\author{Westin King}
\email[corresponding author]{westin\_king@alumni.baylor.edu}

\author{Catherine Yan}
\email{cyan@math.tamu.edu}

\address{Department of Mathematics, Texas A\&M University, College Station, TX 77843, United States}


\begin{abstract}
Classical parking functions can be defined in terms of drivers with preferred parking spaces searching a linear parking lot for an open parking spot. We may consider this linear parking lot as a collection of $n$ vertices (parking spots) arranged in a directed path. We generalize this notion to allow for more complicated ``parking lots'' and define parking functions on arbitrary directed graphs. We then consider a relationship proved by Lackner and Panholzer between parking functions on trees and ``mapping digraphs'' and we show that a similar relationship holds when edge orientations are reversed. 
\end{abstract}
\maketitle

\section{Introduction}\label{sec:introduction}

Parking functions were first defined by Konheim and Weiss \cite{KONHEIM1966} during their study of the linear probing solution to collisions on hash tables. The authors described a sequence of $n$ drivers attempting to park randomly along a one-way street. If the spot a driver attempts to park in is occupied, she drives to the next available spot and parks. As we are interested in the number of ways such a procedure results in all $n$ drivers parking, we may consider the initial checked spot as a preferred parking place, rather than a randomly chosen spot. Let $s \in [n]^n$ and consider a directed path with vertex set $[n]$ and edge orientations $i \rightarrow i+1$. One-by-one the drivers attempt to park according to the following process:

\begin{itemize}
\item[1)] Driver $i$ begins at vertex $s_i$. 
\item[2)] If the current vertex is unoccupied, the driver parks there. If it is occupied, the driver drives to the next vertex, following edge orientation, and repeats Step 2.
\item[3)] If she parks, the process continues with driver $i+1$ attempting to park at vertex $s_{i+1}$. Otherwise, the process terminates.
\end{itemize}

A sequence $s$ such that all $n$ drivers successfully park is called a \emph{classical parking function} of length $n$. With some consideration, one can see that as long as there is no $i$ such that too many drivers prefer a vertex from $\{i,i+1,i+2,\hdots, n\}$, then all drivers will park. This is formulated in an alternative definition of parking functions:

\begin{definition}[Parking Functions]\label{def:classical}
A \emph{classical parking function} of length $n$ is a sequence $s \in [n]^n$ such that for all $i \in [n]$,
\[
|\{j : s_j \geq i\}| \leq n-i+1.
\]
Additionally, we may consider the case when $m \leq n$ drivers attempt to park. We call these $(n,m)$-parking functions and the definition is the same as in the classical case for $s \in [n]^m$.

\end{definition}

It is well-known that there are $(n+1)^{n-1}$ parking functions of length $n$, while there are $(n-m+1)(n+1)^{m-1}$ classical $(n,m)$-parking functions \cite{CAMERON2008}. Classical parking functions have appeared throughout combinatorics as chains in the noncrossing lattice, in enumeration of hyperplane arrangements, in noncrossing partitions, and in tree enumeration (see \cite{GESSEL2006,EC2,YAN2015} as well as references herein).

Parking functions have seen many generalizations: $G$-parking functions \cite{POSTNIKOV2004}, $\txbf{u}$-parking functions \cite{KUNG2003}, parking sequences \cite{EHRENBORG2016},  rational parking functions \cite{ARMSTRONG2016}, and those defined on tree-shaped parking lots \cite{BUTLER2017,LACKNER2016}. In this paper, we extend the ``drivers searching for a parking spot'' analogy from the trees in \cite{LACKNER2016} to general digraphs and give a description that generalizes the set definition of the classical parking functions in Section \ref{sec:definitions}. In Sections \ref{sec:lackner} and \ref{sec:comparing}, we closely follow the results of Lackner and Panholzer \cite{LACKNER2016} and show that many of their theorems have analogues when the edge orientations of the tree and mapping digraphs are reversed. Finally, we conclude with some new research directions and propose extensions to the concepts of increasing and prime parking functions.

This paper is, in part, an expansion of Sections 2 and 4 of \cite{KINGFPSAC2018}.

\section{Parking Functions on Digraphs}\label{sec:definitions}

One interpretation of classical parking functions is of drivers attempting to park in a preferred spot along a street and parking in the first available spot they find afterwards. We may extend this notion to general digraphs by allowing driver to choose which out-edge to travel along. More formally,

\begin{definition}[Parking Process]\label{def:parkingprocess}
Pick $n,m$ such that $0 \leq m \leq n$. Let $s \in [n]^m$ and $D$ be a digraph with vertex set $[n]$. One-by-one $m$ drivers attempt to park according to the following process:

\begin{itemize}
\item[1)] Driver $i$ begins at vertex $s_i$. 
\item[2)] If the current vertex is unoccupied, the driver parks there. If it is occupied, the driver chooses a vertex in the neighborhood of the current one and drives there.
\item[3)] The driver repeats step 2) until she either parks, and the next driver enters, or is unable to find an available parking space, and the process terminates.
\end{itemize}
\end{definition}

When the maximum outdegree of a vertex in $D$ is 1, the parking process is deterministic. In the general case, however, drivers must ``choose" which edge to take in search of a parking spot. Our interest lies in the possibility of all drivers parking, so we give the following definition as our parking function generalization:

\begin{definition}\label{def:digraphpf}
For a sequence $s \in [n]^m$ and digraph $D$ with vertex set $[n]$, we say that $s$ is a \emph{parking function on $D$} if it is possible for all of the $m$ drivers to park following the parking process. If $s$ is a parking function on $D$, we call the pair $(D,s)$ an $(n,m)$\emph{-parking function}.
\end{definition}

Figure \ref{fig:digraphex} gives an example of an $(n,n)$-parking function. Drivers 2 and 5 are the only drivers who may make a choice of which edge to travel along and all drivers can park as long as at least one of those drivers uses the edge $(1,4)$ during parking. 

\begin{figure}
\begin{center}
\begin{tikzpicture}
\node[circle,draw,inner sep=3pt] (a) at (0,0) {$1$};
\node[circle,draw,inner sep=3pt] (d) at (-1.5,1.5) {$2$};
\node[circle,draw,inner sep=3pt] (e) at (1.5,1.5) {$3$};
\node[circle,draw,inner sep=3pt] (f) at (1.5,-1.5) {$4$};
\node[circle,draw,inner sep=3pt] (g) at (-1.5,-1.5) {$5$};

\draw[arrows={->[scale=1.5]}] (e) -- (d);
\draw[arrows={->[scale=1.5]}] (a) -- (e);
\draw[arrows={->[scale=1.5]}] (a) -- (d);
\draw[arrows={->[scale=1.5]}] (f) -- (g);
\draw[arrows={->[scale=1.5]}] (a) -- (f);
\draw[arrows={->[scale=1.5]}] (d) -- (g);
\end{tikzpicture}
\caption{A digraph with parking function $s=(1,1,3,2,1)$.}
\label{fig:digraphex}
\end{center}
\end{figure}
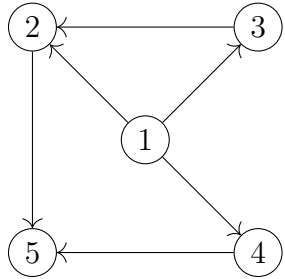

Definition \ref{def:digraphpf} is clearly a generalization of the classical parking function case, but it is sometimes difficult to apply practically, so we now consider an equivalent definition that is more useful. For $i,j \in [n]$, we say $i \preceq_D j$ if and only if there exists a directed path from $i$ to $j$ in $D$. By convention, we will say that $i \preceq_D i$, making $\preceq_D$ a quasiorder on the vertices of $D$. If we wish to consider when $i \neq j$, we say $i \prec_D j$. For vertex $i$, define the set of vertices \emph{reachable from} $i$ as

\[
R_D(i) = \{ j \in [n] : i \preceq_D j \}.
\]
Then for any $A \subseteq [n]$ define the \emph{reachable set of $A$} as

\[
R_D(A) = \bigcup\limits_{i \in A}R_D(i).
\]

\begin{theorem}\label{thm:hall2}
Let $D$ be a digraph with vertex set $[n]$ and $s \in [n]^m$. Let $C=\{C_1, \hdots C_m\}$ be the set of cars, indexed such that $C_i$ prefers spot $s_i$. Then $s$ is a parking function on $D$ if and only if for all $A \subseteq C$ we have 

\[
|A| \leq |\bigcup\limits_{s_i : C_i \in A}R_D(s_i)|.
\]
\end{theorem}

\begin{proof}
If we let $B$ be the bipartite graph with vertex set $C \cup [n]$ where $\{C_i,j\} \in E(B)$ if and only if $s_i \preceq_D j$, then by Hall's Theorem, we are claiming that $s$ is a parking function if and only if there exists a matching on $B$ saturating $C$. If $s$ is a parking function, then park the drivers in some manner. Suppose $C_i$ is parked on $v_i$ for each $i$. Then, since we must have $s_i \preceq_D v_i$, the edge $\{C_i,v_i\}$ is in $B$. These edges define a matching that saturates $C$.

The other direction is not as clear because the parking process requires drivers to park in the first empty spot they find. We use a matching $M = \{ \{C_i,v_i\}\}_{i=1}^m$ to determine a (not necessarily unique) way of parking on $D$. 

The iterative process is the same for each $C_i$, starting at $i=1$. At step $i$, pick any $x$ with $s_i \preceq_D x \preceq_D v_i$ such that there exists a walk $s_i=y_0 \rightarrow y_1 \rightarrow \hdots \rightarrow y_k \rightarrow x$ such that the $y_j$'s are spots occupied by cars in previous iterations. If no such $y_0$ exists, then $x = s_i$. Park $C_i$ in $x$ and delete $\{C_i,v_i\}$ from $M$. If $\{C_j,x\} \in M$ for some $j >i$, then replace this edge with $\{C_j,v_i\}$. Now repeat with $C_{i+1}$.

In each step, the vertex $x$ is the first unoccupied vertex along some walk between $s_i$ and the vertex with which $C_i$ is matched. At least one such $x$ exists because, at the start of step $i$, $C_i$ is matched with a vertex which is not occupied by any car. At the end of step $i$, we know the updated $M$ is a matching saturating $\{C_\ell\}_{\ell>i}$ because we know $s_j \preceq_D v_j = x$ from the edge $\{C_j,v_j\}$ and $x \preceq_D v_i$ by our choice of $x$. Thus, $s_j \preceq_D v_i$, so the edge $\{C_j,v_i\}$ is in $B$.
\end{proof}

Rather than considering subsets of the drivers, we may reformulate the statement in terms of the number of drivers wanting to park in various ``regions'' of the graph (c.f. Definition \ref{def:classical}).

\begin{cor}\label{cor:hall}
Let $D$ be a digraph with vertex set $[n]$ and $s \in [n]^m$. Then $s$ is a parking function on $D$ if and only if for all $B \subseteq [n]$ we have 

\[
|\{C_i : s_i \in R_D(B)\}| \leq |R_D(B)|.
\]
\end{cor}

\begin{proof}
Let $s$ be a parking function on $D$ and $B \subseteq [n].$ Let $A = \{C_i : s_i \in R_D(B)\}$. Because $s$ is a parking function, we know $|A| \leq |\bigcup\limits_{s_i : C_i \in A}R_D(s_i)|$, but also by the definition of $R_D(B)$, we have $\bigcup\limits_{s_i : C_i \in A}R_D(s_i) \subseteq R_D(B)$.

On the other hand, suppose for all $B \subseteq [n]$ we have $|\{C_i : s_i \in R_D(B)\}| \leq |R_D(B)|$, let $A \subseteq C$, and $B = \bigcup\limits_{C_i \in A}\{s_i\}.$ By definition, $R_D(B) = \bigcup\limits_{C_i \in A}R_D(s_i),$ and so 
\[
|A| \leq |\{C_i : s_i \in R_D(B)\}| \leq |R_D(B)| = |\bigcup\limits_{C_i \in A}R_D(s_i)|,
\]
thus $s$ is a parking function.
\end{proof}

From this set definition, it is immediately obvious that the ordering of $s$ does not matter.

\begin{cor}
Let $(D,s)$ be a parking function and $\sigma \in \mathfrak{S}_m$. Then the permuted sequence $s_\sigma = (s_{\sigma(1)},s_{\sigma(2)},\hdots, s_{\sigma(m)})$ is also a parking function on $D$.
\end{cor}

\begin{remark}
The number of distinct $R_D(B)$ is the same as the number of filters of the quasiorder $\preceq_D$, which is, in general, much less than $2^n$.
\end{remark}

In the case of a classical parking function $s$, for any $B \subseteq [n]$ with smallest element $b$, we have that $|R_{\mathcal{P}_{n}}(B)| = |R_{\mathcal{P}_{n}}(b)| = n+1-b$. Thus, for $i \in [n]$, $|\{j : s_j \geq i\}| \leq n+1-i$, as in Definition \ref{def:classical}. In addition, Corollary \ref{cor:hall} generalizes the description of parking functions given in \cite{LACKNER2016}, which we now consider as we turn our attention to rooted trees with edges oriented away from the root.

\section{Parking Functions on Source Trees}\label{sec:lackner}

Lackner and Panholzer \cite{LACKNER2016} and Butler, Graham, and Yan \cite{BUTLER2017} independently generalized the ``drivers searching for a parking space'' interpretations of parking functions to rooted trees with edges oriented towards a root. In this section, we follow the enumerative discussion from \cite{LACKNER2016} and show that several of their theorems have analogous results when edge orientations are reversed. 

If $T$ is a rooted tree with vertex set $[n]$ and edges oriented towards the root, we let $\widetilde{T}$ be the tree obtained by reversing the orientation of all the edges. In our pictures, the root will be the top-most vertex. We call these \emph{sink} and \emph{source} trees, matching whether the root is a sink or source vertex. Similarly, if $M_f$ is the digraph obtained from $f: [n] \rightarrow [n]$ by letting $V(M_f)= [n]$ with edge set $E = \{(i,f(i))\}_{i=1}^n$, we let $\widetilde{M}_f$ be the digraph obtained from $M_f$ by reversing edge orientations. That is, $E(\widetilde{M}_f) = \{(f(i),i)\}_{i=1}^n$. We will call these \emph{mapping} and \emph{inverse mapping} digraphs, respectively. We note that digraphs in each of thse familes have a single cycle on each connected component Define the sets

\[
\mathcal{T}_n = \{T: |V(T)|=n \tx{ and $T$ is a rooted sink tree}\},
\]
\[
\mathcal{M}_n = \{M: M \tx{ is the mapping digraph of some }f:[n] \rightarrow [n]\}.
\]
We similarly define $\widetilde{\mathcal{T}}_n$ and $\widetilde{\mathcal{M}}_n$ for the source trees and inverse mapping digraphs.


Figure \ref{fig:sourceex} gives an example of a source tree and an inverse mapping digraph along with an $s \in [7]^7$ that is a parking function on both. 

\begin{figure}[h]
\begin{center}
\begin{tikzpicture}
\node[circle,draw,inner sep=1pt] (a) at (-1,0) {$3$};
\node[circle,draw,inner sep=1pt] (b) at (-1.8,-0.8) {$6$};
\node[circle,draw,inner sep=1pt] (c) at (-0.2,-0.8) {$1$};
\node[circle,draw,inner sep=1pt] (d) at (-1,-1.6) {$2$};
\node[circle,draw,inner sep=1pt] (e) at (-0.2,-1.6) {$7$};
\node[circle,draw,inner sep=1pt] (f) at (0.6,-1.6) {$4$};
\node[circle,draw,inner sep=1pt] (g) at (0.6,-2.4) {$5$};
\node (h) at (2,1) {$s = (2,3,4,1,3,5,1)$};

\draw[arrows={->[scale=1.5]}] (a) -- (b);
\draw[arrows={->[scale=1.5]}] (a) -- (c);
\draw[arrows={->[scale=1.5]}] (c) -- (d);
\draw[arrows={->[scale=1.5]}] (c) -- (e);
\draw[arrows={->[scale=1.5]}] (c) -- (f);
\draw[arrows={->[scale=1.5]}] (f) -- (g);


\node[circle,draw,inner sep=1pt] (j) at (5,0) {$3$};
\node[circle,draw,inner sep=1pt] (k) at (4.2,-0.8) {$6$};
\node[circle,draw,inner sep=1pt] (l) at (6.4,-0.8) {$1$};
\node[circle,draw,inner sep=1pt] (m) at (5.6,-1.6) {$2$};
\node[circle,draw,inner sep=1pt] (n) at (6.4,-1.6) {$7$};
\node[circle,draw,inner sep=1pt] (o) at (7.2,-1.6) {$4$};
\node[circle,draw,inner sep=1pt] (p) at (4.2,-1.6) {$5$};

\draw[arrows={->[scale=1.5]}] (j) -- (k);
\draw[arrows={->[scale=1.5]}] (j) to [out=330,in=45,loop] (j);
\draw[arrows={->[scale=1.5]}] (l) -- (o);
\draw[arrows={->[scale=1.5]}] (l) -- (m);
\draw[arrows={->[scale=1.5]}] (l) -- (n);
\draw[arrows={->[scale=1.5]}] (o) to [out=45,in=30] (l);
\draw[arrows={->[scale=1.5]}] (p) to [out=330,in=45,loop] (p);

\end{tikzpicture}
\caption{A source tree and inverse mapping digraph with a common parking function.}
\label{fig:sourceex}
\end{center}
\end{figure}
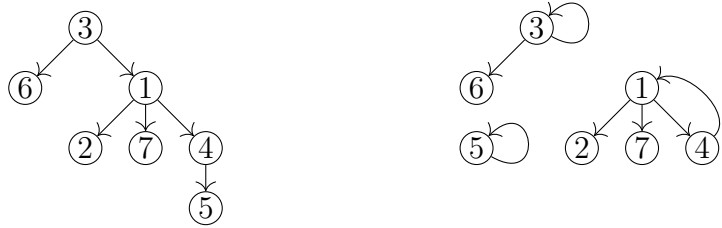

For source trees, because the indegree of a node is at most 1, for two distinct vertices, $u$ and $v$, the sets $R_{\widetilde{T}}(u)$ and $R_{\widetilde{T}}(v)$ are either disjoint or one contains the other. Thus, the $2^n$ inequalities in Corollary \ref{cor:hall} reduce to only $n$ independent inequalities. 

We briefly introduce some notation out of convenience. We call $u$ the \emph{parent} of $v$ if the two are adjacent and $u$ lies on the unique path between the root and $v$. Here, we do not consider edge orientation. Further, for a vertex $u$ of $\widetilde{T}$, let $\widetilde{T}_u$ be the subtree induced by the set $R_{\widetilde{T}}(u)$. Thus, Corollary \ref{cor:hall} can be restated as:

\begin{cor}\label{cor:shortcut}
Let $\widetilde{T}$ be a source tree and $s \in [n]^m$. Then $s$ is an $(n,m)$-parking function on $\widetilde{T}$ if and only if for all $u \in [n]$ we have 
\[
|\{i : s_i \in \widetilde{T}_u \}| \leq |\widetilde{T}_u|.
\]
\end{cor}

Additionally, for digraph $D$ with vertex set $[n]$, let 
\[
P(D,m) = |\{(D,s) : s\in [n]^m \tx{ is a parking function on } D\}|.
\] 
We first study the extremal values of $P(\widetilde{T},m)$.

\begin{prop}
Let $\widetilde{T}$ be a source tree, $u$ a non-root vertex, $v$ the parent of $u$, and $w$ such that $v \preceq_{\widetilde{T}} w$ and $w \notin \widetilde{T}_u$. Let $\widetilde{T}'$ be the tree obtained by removing the edge $(v,u)$ and adding the edge $(w,u)$. Then $P(\widetilde{T},m) \leq P(\widetilde{T}',m)$.
\end{prop}

\begin{proof}
To clarify which tree we are considering, we denote by $x'$ the vertex in $\widetilde{T}'$ with label $x$. Let $(\widetilde{T},s)$ be an $(n,m)$-parking function. By Corollary \ref{cor:hall}, we must check $|\{i : s_i \in \widetilde{T}'_{x'} \}| \leq |\widetilde{T}'_{x'}|$ for all $x \in V(\widetilde{T}')$. By the construction of $\widetilde{T}'$, $|\{i : s_i \in \widetilde{T}'_{y'} \}| = |\{i : s_i \in \widetilde{T}_y \}|$ and $|\widetilde{T}'_{y'}| = |\widetilde{T}_y|$ for all $y'$ not satisfying $v' \prec_{\widetilde{T}'} y' \preceq_{\widetilde{T}'} w'$.

Therefore, let $y'$ be a vertex satisfying $v' \prec_{\widetilde{T}'} y' \preceq_{\widetilde{T}'} w'$. We thus have:

\begin{align*}
|\{i : s_i \in \widetilde{T}'_{y'} \}| &= |\{i : s_i \in \widetilde{T}_y \}| + |\{i : s_i \in \widetilde{T}_u \}| \\
&\leq |\widetilde{T}_y| + |\widetilde{T}_u| \\
&= |\widetilde{T}'_{y'}|.
\end{align*}
\end{proof}

As a result, we obtain an upper and lower bound on $P(\widetilde{T},m)$.

\begin{cor}\label{cor:extreme}
The number of $(n,m)$-parking functions is maximized when $\widetilde{T}$ is a path and minimized when $\widetilde{T}$ is a star, meaning

\[
\sum\limits_{i=0}^{m}{m \choose i}(n-1)^{\underline{m-i}} \leq P(\widetilde{T},m) \leq (n-m+1)(n+1)^{m-1}.
\]
\end{cor}

\begin{proof}
For a path, the parking functions, up to vertex labeling, are classical parking functions, and thus number $(n-m+1)(n+1)^{m-1}$.

On a star, for $0 \leq i \leq m$, when $i$ drivers prefer the root, there are ${n-1 \choose m-i}$ ways to choose the preferred non-root vertices, ${m \choose i}$ ways to place the drivers preferring the root in $s$, and $(m-i)!$ ways to order the drivers preferring non-roots in $s$.
\end{proof}

In the case of sink trees, the maximum and minimum number of parking functions also occur on paths and stars, respectively. From \cite{LACKNER2016}, we have
\[
n^{\underline{m}} + {m \choose 2}(n-1)^{\underline{m-1}} \leq P(T,m) \leq (n-m+1)(n+1)^{m-1}.
\]
If $T$ is a star, it is not immediately clear for an arbitrary choice of $(n,m)$ which of $P(T,m)$ and $P(\widetilde{T},m)$ is larger. However, we can say after inspecting the formulae that if $m$ is 0 or $1$ then the two are equal and if $n \geq 3$ then $P(T,3)$ is larger. What happens when $3 < m < n$ remains open, but we can determine which is larger when $m=n$ for all $T$.

\begin{theorem}\label{thm:treenumber}
Let $T \in \mathcal{T}_n$. Then
\[
P(T,n) \leq P(\widetilde{T},n)
\]
with equality if and only if $T$ is a path.
\end{theorem}

\begin{proof}
Let $(T,s)$ be a parking function on sink tree $T$. We give a process to determine an involution $\tau \in \mathfrak{S}_n$ such that $\tau(s) = (\tau(s_1),\tau(s_2),\hdots,\tau(s_n))$ is a parking function on the source tree $\widetilde{T}$. Park cars on $T$ following the parking procedure, highlighting an edge if it is used by a driver after failing to park at her preferred spot. Since $T$ is a sink tree, each vertex has outdegree at most 1, so parking is deterministic. We define $\tau$ by individually considering the components connected by highlighted edges. So without loss of generality, we may assume that every edge in $T$ is highlighted.

We define a collection of length $\geq 2$ ``paths'' in $T$, one for each leaf, whose vertices form a partition of the vertices of $T$. The purpose of these ``paths'' is to identify a section of the tree where we can ``flip'' the edge orientations and driver preferences and still guarantee each driver a spot to park. 

Let $\{v_i\}$, for $1 \leq i \leq k$ be the leaves of $T$ indexed such that $v_i < v_{i+1}$. We recursively define the sets $P_i$: the set $P_i$ is the smallest set of vertices of the path between $v_i$ and the root such that no vertices from $P_j, j<i,$ are in $P_i$, there are $|P_i|$ drivers preferring $P_i$, and all vertices of $P_i$ are connected to $v_i$ through vertices in $\{P_j\}_{j=1}^i$. For example, on the left tree of Figure \ref{fig:treeinversion}, we have sets $P_1 = \{1,2,3,5\}$ and $P_2 = \{4,6\}$. Two drivers prefer the leaf $\{1\}$, three drivers prefer the vertices $\{1,2\}$, four drivers prefer $\{1,2,3\}$ and $\{1,2,3,5\}$, so the latter is $P_1$. When determining $P_2$, we skip over vertices in previously-chosen paths, in this case the vertex labeled 5. 

\begin{figure}[h]
\begin{center}
\begin{tikzpicture}
\node[circle,draw,inner sep=1pt] (a) at (0,0) {$6$};
\node[circle,draw,inner sep=1pt] (b) at (0,-0.8) {$5$};
\node[circle,draw,inner sep=1pt] (c) at (0.8,-1.6) {$4$};
\node[circle,draw,inner sep=1pt] (d) at (-0.8,-1.6) {$3$};
\node[circle,draw,inner sep=1pt] (e) at (-0.8,-2.4) {$2$};
\node[circle,draw,inner sep=1pt] (f) at (-0.8,-3.2) {$1$};
\node (g) at (0,0.5) {$p = (1,4,4,2,1,3)$};

\draw[arrows={->[scale=1.5]}] (b) -- (a);
\draw[arrows={->[scale=1.5]}] (c) -- (b);
\draw[arrows={->[scale=1.5]}] (d) -- (b);
\draw[arrows={->[scale=1.5]}] (e) -- (d);
\draw[arrows={->[scale=1.5]}] (f) -- (e);

\node (h) at (2.5,-1.5) {$\longrightarrow$};

\node[circle,draw,inner sep=1pt] (h) at (5,0) {$6$};
\node[circle,draw,inner sep=1pt] (i) at (5,-0.8) {$5$};
\node[circle,draw,inner sep=1pt] (j) at (5.8,-1.6) {$4$};
\node[circle,draw,inner sep=1pt] (k) at (4.2,-1.6) {$3$};
\node[circle,draw,inner sep=1pt] (l) at (4.2,-2.4) {$2$};
\node[circle,draw,inner sep=1pt] (m) at (4.2,-3.2) {$1$};
\node (n) at (5,0.5) {$\tau(p) = (5,6,6,3,5,2)$};

\draw[arrows={->[scale=1.5]}] (h) -- (i);
\draw[arrows={->[scale=1.5]}] (i) -- (j);
\draw[arrows={->[scale=1.5]}] (i) -- (k);
\draw[arrows={->[scale=1.5]}] (k) -- (l);
\draw[arrows={->[scale=1.5]}] (l) -- (m);

\end{tikzpicture}
\caption{Constructing a parking function on $\widetilde{T}$ from one on $T$. $\tau = (15)(23)(46)$}
\label{fig:treeinversion}
\end{center}
\end{figure}
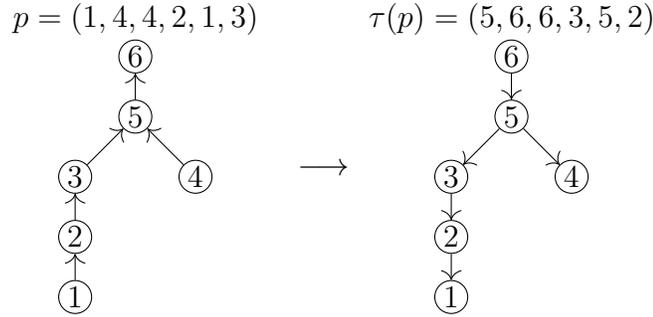

The collection $\{P_i\}_{i=1}^k$ must partition the vertices of $T$. Suppose it does not and let $v \notin P_i$ for any $i$ be such that $v$ is the only vertex in $T_v$, the subtree rooted at $v$, with this property. The vertex $v$ is not a leaf of $T$, as all leaves are in the sets $\{P_i\}_{i=1}^k$ by construction, and thus is the terminus of at least one edge, $(u,v)$. Since none of the ``paths'' corresponding to leaves of $T_v$ contain $v$ and because all cars can park, all cars preferring spots $w \preceq_T u$ can park without occupying $v$. This means the edge $(u,v)$ is not used by any driver after failing to park in her preferred spot, which contradicts our assumption that this was true of all edges. Therefore, such a $v$ can not exist.

For each $i$, let $n_i = |P_i|$ and label the elements of $P_i$ by $w_{i,j}$ such that $w_{i,1} = v_i \preceq_T w_{i,2} \preceq_T \hdots \preceq_T w_{i,n_{i}}$. Finally, we define $\tau(w_{i,j}) = w_{i,n_{i}+1-j}$. This reverses the driver preference along the ``path'' so that when the edge orientation is flipped for $\widetilde{T}$, the drivers may park as they did on $T$.

We can recover the $P_i$ from $(\widetilde{T},\tau(s))$ using the exact same method. Thus, we can invert the process. On the right tree of Figure \ref{fig:treeinversion}, zero drivers prefer $\{1\}$, one driver prefers $\{1,2\}$, two drivers prefer $\{1,2,3\}$, and four drivers $\{1,2,3,5\}$, so this is $P_1$. For $P_2$, no drivers prefer $\{4\}$, we skip over 5 as it is already in $P_1$, and two drivers prefer $\{4,6\}$. 

If $T$ is not a directed path, then this process is not surjective because the parking function in which all $n$ drivers prefer the root of $\widetilde{T}$ is not obtainable in this manner as at least one driver prefers each leaf vertex.
\end{proof}

Define the following for $n \geq 1$:

\[
F_{n,m} =  \sum\limits_{T \in \mathcal{T}_{n}}P(T,m), \tx{ and } M_{n,m} = \sum\limits_{M \in \mathcal{M}_{n}}P(M,m).
\]
Similarly, define $\widetilde{F}_{n,m}$ and $\widetilde{M}_{n,m}$ for source trees and inverse mappings. Summing over all $T \in \mathcal{T}_n$ gives us

\begin{cor}
For $n \geq 1$,
\[
F_{n,n} \leq \widetilde{F}_{n,n},
\]
with equality only when $n \in \{1,2\}$.
\end{cor}

As we previously mentioned, when $m<n$, the result of Theorem \ref{thm:treenumber} does not necessarily hold. For the tree in Figure \ref{fig:treecounter}, $P(T,2) = 15$, as any sequence except $(4,4)$ parks. However, $P(\widetilde{T},2)=14$ as neither $(1,1)$ nor $(2,2)$ are parking functions.

\begin{figure}[h]
\begin{center}
\begin{tikzpicture}
\node[circle,draw,inner sep=1pt] (a) at (0,0) {$4$};
\node[circle,draw,inner sep=1pt] (b) at (0,-1) {$3$};
\node[circle,draw,inner sep=1pt] (c) at (-0.5,-2) {$1$};
\node[circle,draw,inner sep=1pt] (d) at (0.5,-2) {$2$};

\draw (a) -- (b);
\draw (b) -- (c);
\draw (b) -- (d);
\end{tikzpicture}
\caption{A tree for which $P(T,2) > P(\widetilde{T},2)$.}
\label{fig:treecounter}
\end{center}
\end{figure}
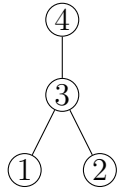

\begin{remark}
The extremal values for $P(\widetilde{M}_f,m)$ are less interesting than the corresponding values on trees. When $\widetilde{M}_f$ is a cycle, $P(\widetilde{M}_f,m)$ is maximal (as all $n^m$ sequences can park) and is minimal when $f=id$. Thus
\[
m! \leq P(\widetilde{M}_f,m) \leq n^m.
\]
In fact, the same is true for $P(M_F,m)$.
\end{remark}

\section{Comparing $P(\widetilde{T},m)$ and $P(\widetilde{M}_f,m)$.}\label{sec:comparing}

Lackner and Panholzer \cite{LACKNER2016} proved $n \cdot  F_{n,m} = M_{n,m}$. In fact, this relationship still holds when the edge orientations are reversed. We prove the claim when $m=n$, then we will show the more general case. While the overall idea of the proofs are similar to their counterparts in \cite{LACKNER2016}, there are some technical differences in dealing with the source trees. Because, on sink trees, the spot in which each driver parked was well-defined, the authors of \cite{LACKNER2016} were able to identify edges in the digraphs that were not used during parking and thus could be freely manipulated. In our case, as the drivers no longer necessarily have a unique walk along which to search for a parking spot, we must instead identify edges that are not necessary for \emph{some} successful parking. This is simple enough on trees using the characterization of Corollary \ref{cor:hall}, but it is not immediately clear for inverse mapping digraphs. So, we first prove that at least one cycle edge on each component of an inverse mapping digraph is not needed for parking.

\begin{lemma}\label{lem:cycles}
Let $(\widetilde{M}_f,s)$ be a parking function. Then there exists at least one edge in each cycle of $\widetilde{M}_f$ that can be deleted such that all drivers can still park.
\end{lemma}

\begin{proof}
We induct on the number of vertices in a cycle. Let $(\widetilde{M}_f,s)$ be an $(n,n)$-parking function. Without loss of generality, we may assume there is only one component of $\widetilde{M}_f$ and thus a unique cycle in the graph. If only one vertex is in the cycle, then there is an edge of the form $(u,u)$ which is useless for parking and may be deleted.

Now suppose the cycle has length $r > 1$. Furthermore, suppose without loss of generality that the vertices of the cycle are labeled by $[r]$, $f(r)=1,$ and $f(i)=i+1$ for $i \in [r-1]$. Let $M$ be the graph obtained by deleting all cycle edges. Define for $1 \leq i \leq r,$ $V_i := |R_M(i)|$ and $\alpha_i := |\{j:s_j \in R_M(i)\}|$. These are the numbers of vertices in and drivers preferring the subtree induced by the vertex $i$ along with all $i \preceq_{\widetilde{M}_f} v$ for $v$ non-cycle vertices in $\widetilde{M}_f$.

If $\alpha_i < V_i$ for all $i \in [r]$, then there are strictly fewer than $n$ drivers attempting to park, contradicting the assumption that $(\widetilde{M}_f,s)$ is an $(n,n)$-parking function. Hence, we know $\alpha_i \geq V_i$ for some $i$. Since $s$ is a parking function on $\widetilde{M}_f$, at least one driver prefers $i$ (otherwise, too many drivers prefer non-cycle vertices). We construct an $(n-1,n-1)$-parking function by contracting the edge $(i,i-1)$ to identify the vertices $i-1$ and $i$ as a single vertex (with label $i-1$) to form the digraph $M'$, deleting the first instance of $i$ from $s,$ and changing all others to $i-1$ to form the sequence $s'$. For non-cycle vertex $v$, $|R_{\widetilde{M}_f}(v)| = |R_{M'}(v)|$, as are the number of drivers preferring each set. If $v$ is instead a cycle vertex, then $n=|R_{\widetilde{M}_f}(v)| = |R_{M'}(v)|+1$, while $n$ drivers prefer $R_{\widetilde{M}_f}(v)$ and $n-1$ drivers prefer $R_{M'}(v)$. Thus, $(M',s')$ is a parking function with $r-1$ vertices in the cycle. By the inductive hypothesis, there exists an edge $e$ that can be deleted from $M'$. We claim this same edge in $\widetilde{M}_f$ is not necessary for parking via $s$.

Let $T$ be the digraph obtained by deleting $e$ from $\widetilde{M}_f$ and $T'$ be obtained by deleting $e$ from $M'$. Since $(T',s')$ is a parking function, we know for any $v \in T'$, we have $|\{j : s'_j \in R_{T'}(v)\}| \leq |R_{T'}(v)|$. We now check the vertices of $T$ to determine if $(T,s)$ is a parking function.

\txbf{Case 1}: $i-1 \prec_T v$. We have $$|\{j : s_j \in R_T(v)\}| = |\{j : s_j \in R_{T'}(v)| \leq |R_{T'}(v)| =|R_{T}(v)|.$$

\txbf{Case 2}: $v \prec_T i$. Then, $$|\{j : s_j \in R_T(v)\}| = |\{j : s_j \in R_{T'}(v)|+1 \leq |R_{T'}(v)|+1 =|R_{T}(v)|.$$

\txbf{Case 3}: $v=i$ gives $$|\{j : s_j \in R_T(i)\}| = |\{j : s_j \in R_{T'}(i-1)|+1 \leq |R_{T'}(i-1)|+1 =|R_{T}(i)|.$$

\txbf{Case 4}: $v=i-1$. Using the fact that $-\alpha_i \leq -V_i$, we know

\begin{align*}
|\{j : s_j \in R_T(i-1)\}| &= |\{j : s_j \in R_{T'}(i-1)|+1-\alpha_i \\
& \leq |R_{T'}(i-1)|+1-V_i \\
&= (|R_T(i-1)| + V_i -1) +1 - V_i \\
&= |R_T(i-1)|.
\end{align*}

\txbf{Case 5}: all other $v$. Since $v$ is not a cycle vertex in $\widetilde{M}_f$ and $s$ is a parking function on $\widetilde{M}_f$, the deleting of $e$ does not affect the reachable set of $v$. Thus, 
\[
|\{j : s_j \in R_T(v)\}| = |\{j : s_j \in R_{\widetilde{M}_f}(v)\}| \leq |R_{\widetilde{M}_f}(v)| = |R_{T}(v)|.
\]
So $(T,s)$ is indeed a parking function and we know the edge $e$ is not necessary for parking on $\widetilde{M}_f$.
\end{proof}

Figure \ref{fig:contraction} gives an example of the contraction to $M'$ with $i=3$. We identify a deletable $e$ on $M'$, which gives a deletable $e$ on $\widetilde{M}_f$. 

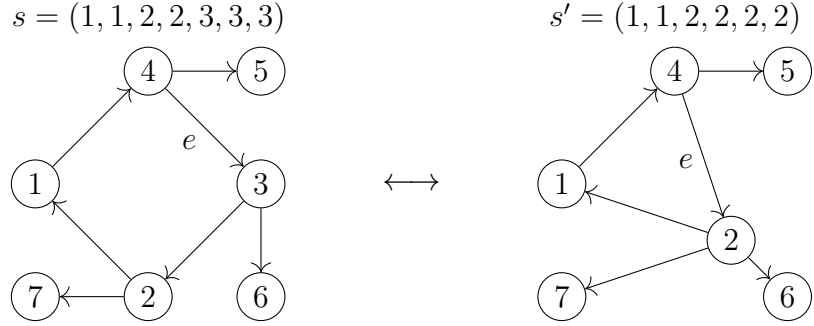
\begin{figure}
\begin{center}
\begin{tikzpicture}
\node[circle,draw,inner sep=3pt] (a) at (0,1.5) {$4$};
\node[circle,draw,inner sep=3pt] (b) at (1.5,0) {$3$};
\node[circle,draw,inner sep=3pt] (c) at (0,-1.5) {$2$};
\node[circle,draw,inner sep=3pt] (d) at (-1.5,0) {$1$};
\node[circle,draw,inner sep=3pt] (e) at (1.5,1.5) {$5$};
\node[circle,draw,inner sep=3pt] (f) at (1.5,-1.5) {$6$};
\node[circle,draw,inner sep=3pt] (g) at (-1.5,-1.5) {$7$};

\draw[arrows={->[scale=1.5]}] (a) -- (b);
\draw[arrows={->[scale=1.5]}] (a) -- (e);
\draw[arrows={->[scale=1.5]}] (b) -- (c);
\draw[arrows={->[scale=1.5]}] (b) -- (f);
\draw[arrows={->[scale=1.5]}] (c) -- (d);
\draw[arrows={->[scale=1.5]}] (c) -- (g);
\draw[arrows={->[scale=1.5]}] (d) -- (a);

\node (i) at (.55,.55) {$e$};
\node (j) at (7.15,0.3) {$e$};

\node (l) at (0,2.2) {$s=(1,1,2,2,3,3,3)$};
\node (m) at (3.5, 0) {$\longleftrightarrow$};
\node (t) at (7,2.2) {$s'=(1,1,2,2,2,2)$};

\node[circle,draw,inner sep=3pt] (n) at (7,1.5) {$4$};
\node[circle,draw,inner sep=3pt] (o) at (7.75,-0.75) {$2$};
\node[circle,draw,inner sep=3pt] (p) at (5.5,0) {$1$};
\node[circle,draw,inner sep=3pt] (q) at (8.5,1.5) {$5$};
\node[circle,draw,inner sep=3pt] (r) at (8.5,-1.5) {$6$};
\node[circle,draw,inner sep=3pt] (s) at (5.5,-1.5) {$7$};

\draw[arrows={->[scale=1.5]}] (n) -- (o);
\draw[arrows={->[scale=1.5]}] (o) -- (p);
\draw[arrows={->[scale=1.5]}] (p) -- (n);
\draw[arrows={->[scale=1.5]}] (n) -- (q);
\draw[arrows={->[scale=1.5]}] (o) -- (r);
\draw[arrows={->[scale=1.5]}] (o) -- (s);
\end{tikzpicture}
\caption{Identifying a deletable edge $e$ by contracting $(3,2)$.}
\label{fig:contraction}
\end{center}
\end{figure}

We now use Lemma \ref{lem:cycles} to prove

\begin{theorem}
For $n \geq 1$, we have the relationship
\[
n \cdot \widetilde{F}_{n,n} = \widetilde{M}_{n,n}.
\]
\end{theorem}

\begin{proof}[Construction of the bijection]
Let $(\widetilde{T},s)$ be a parking function for source tree $\widetilde{T}$, and pick $v \in V(\widetilde{T})$. We define a bijection $\psi$ such that $\psi\lp(\widetilde{T},s,v)\rp=(\widetilde{M}_f,s)$ for some appropriate inverse mapping digraph $\widetilde{M}_f$, constructed by identifying edges in $\widetilde{T}$ that can be manipulated without affecting the ability of the cars to park. The sequence $s$ will not change.

Let $(u,w)$ be an edge in $\widetilde{T}$. If $|\{i:s_i \in \widetilde{T}_w\}| = |\widetilde{T}_w|$, then for any successful parking, no car may cross $(u,w)$ as otherwise too many cars would attempt to park in the subtree $\widetilde{T}_w$. Additionally, at least one driver must prefer $w$, as one driver must park in $w$ and no driver may use the edge $(u,w)$. These two observations will allow us to select edges to manipulate in $\widetilde{T}$.

Consider the path $\text{root}(\widetilde{T}) = v_1 \rightarrow v_2 \rightarrow \hdots \rightarrow v_k = v$ for some $k \geq 1$. We first identify the edges that are freely manipulatable, then we use the order of $s$ to choose a subset of those. For $1 \leq i \leq k$, let $v_i \in A$ if and only if $|\{j : s_j \in \widetilde{T}_{v_{i}}\}| = |\widetilde{T}_{v_{i}}|$. Since $\widetilde{T}_{v_{1}} = \widetilde{T}$, $A$ is nonempty. By the second observation above, all $v_i \in A$ appear as preferences in $s$. The edge $(v_{i-1},v_i)$ is not used by any driver, so we may manipulate it (even delete it) without affecting parking. Next, for the $v_i \in A$, we define the rank $d(v_i)$ to be the index of the first appearance of $v_i$ in $s$. We now let $B \subseteq A$ be given by the elements $\{v_i : \forall j<i, d(v_j) < d(v_i)\}.$ That is, the elements of $B$ are those in $A$ that appear in $s$ after their ancestors from $A$. Note that $B \neq \emptyset$ as root$(\widetilde{T})=v_1 \in B$. 

Consider the unique sequence $\{v_{i_{j}}\}_{j=1}^{|B|}$ such that $v_{i_j} \in B$ and $i_j < i_{j+1}$. For $j > 1$, remove the edge $(v_{i_{j}-1},v_{i_{j}})$ and add the edge $(v_{i_{j}-1},v_{i_{(j-1)}})$. Finally, add the edge $(v,v_{i_{|B|}})$ (if $v$ is the root, this will be a loop as then $v=v_1$). The resulting graph is an inverse mapping digraph, $\widetilde{M}_f$, where $f(i)$ is the unique $j$ such that $(j,i)$ is an edge. In particular, the edges that were manipulated and the one that was added are $\{(f(v_{i_{j}}),v_{i_{j}})\}_{j=1}^{|B|}$.
\end{proof}

\begin{figure}[h]
\begin{center}
\begin{tikzpicture}
\node[circle,draw,inner sep=1pt] (a) at (-1,0) {$3$};
\node[circle,draw,inner sep=1pt] (b) at (-1.8,-0.8) {$6$};
\node[circle,draw,inner sep=1pt] (c) at (-0.2,-0.8) {$1$};
\node[circle,draw,inner sep=1pt] (d) at (-1,-1.6) {$2$};
\node[circle,draw,inner sep=1pt] (e) at (-0.2,-1.6) {$7$};
\node[circle,draw,inner sep=1pt] (f) at (0.6,-1.6) {$4$};
\node[circle,draw,inner sep=1pt,fill=yellow] (g) at (0.6,-2.4) {$5$};
\node (h) at (2,1) {$s = (2,3,4,1,3,5,1), v=5$};

\draw[arrows={->[scale=1.5]}] (a) -- (b);
\draw[arrows={->[scale=1.5]},dotted] (a) -- (c);
\draw[arrows={->[scale=1.5]}] (c) -- (d);
\draw[arrows={->[scale=1.5]}] (c) -- (e);
\draw[arrows={->[scale=1.5]},dotted] (c) -- (f);
\draw[arrows={->[scale=1.5]},dotted] (f) -- (g);

\node (i) at (2.5,-1.5) {$\longrightarrow$};

\node[circle,draw,inner sep=1pt] (j) at (5,0) {$3$};
\node[circle,draw,inner sep=1pt] (k) at (4.2,-0.8) {$6$};
\node[circle,draw,inner sep=1pt] (l) at (5.8,-0.8) {$1$};
\node[circle,draw,inner sep=1pt] (m) at (5,-1.6) {$2$};
\node[circle,draw,inner sep=1pt] (n) at (5.8,-1.6) {$7$};
\node[circle,draw,inner sep=1pt] (o) at (6.6,-1.6) {$4$};
\node[circle,draw,inner sep=1pt] (p) at (6.6,-2.4) {$5$};

\draw[arrows={->[scale=1.5]}] (j) -- (k);
\draw[arrows={->[scale=1.5]}] (j) to [out=330,in=45,loop] (j);
\draw[arrows={->[scale=1.5]}] (l) -- (o);
\draw[arrows={->[scale=1.5]}] (l) -- (m);
\draw[arrows={->[scale=1.5]}] (l) -- (n);
\draw[arrows={->[scale=1.5]}] (o) to [out=45,in=30] (l);
\draw[arrows={->[scale=1.5]}] (p) to [out=330,in=45,loop] (p);

\end{tikzpicture}
\caption{Turning a source tree into an inverse mapping digraph.}
\label{fig:sourcebijection}
\end{center}
\end{figure}
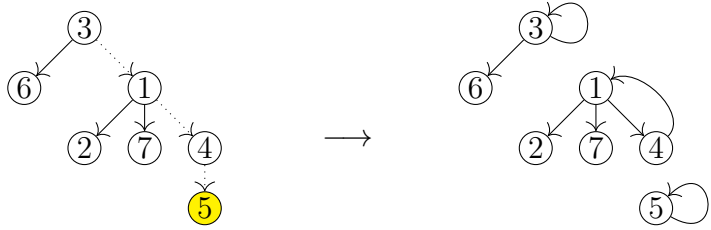

Figure \ref{fig:sourcebijection} gives an example of $\psi$. In it, $A = \{1,3,4,5\}$ and $B =\{1,3,5\}$, with $v_{i_{1}}=3$, $v_{i_{2}}=1$, and $v_{i_{3}}=5$.

\begin{remark}\label{rmk:B}
In each component, the cycle vertex appearing in $B$ has the highest rank of all other cycle vertices in that component. Further, if an edge was necessary for parking on $\widetilde{T}$, it is still necessary for parking on $\widetilde{M}_f$.
\end{remark}

For the inverse, we know by Lemma \ref{lem:cycles} that at least one edge in each cycle of the graph is not necessary for parking. We let the set $\hat{A}$ be the set of vertices in the cycles of $\widetilde{M}_f$ that are the terminal vertices of an edge that is not necessary for parking. Define $\hat{B}\subseteq \hat{A}$ as the set of vertices which have the highest rank in each cycle. By Remark \ref{rmk:B}, if $(\widetilde{M}_f,s) = \psi \lp (\widetilde{T},s,v)\rp$, we know $B =\hat{B}$. Label these elements of $\hat{B}$ by $\{b_i\}_{i=1}^{|\hat{B}|}$ such that $d(b_1) < d(b_2) < \hdots < d(b_{|\hat{B}|})$. For $1 \leq i \leq |\hat{B}|-1,$ remove the edge $\lp f(b_i),b_i \rp$ and add the edge $\lp f(b_i), b_{i+1} \rp$. Finally, delete the edge $\lp f(b_{|\hat{B}|}) , b_{|\hat{B}|} \rp$ and mark $f(b_{|\hat{B}|})$. The resulting tree is $\widetilde{T}$, so $\psi^{-1}\lp (\widetilde{M}_f , s) \rp = (\widetilde{T}, s, f(b_{|\hat{B}|})).$

As promised, we can extend this result to $(n,m)$-parking functions.

\begin{theorem}
Let $n \in \N$ and $0 \leq m \leq n$. Then
\[
n\cdot \widetilde{F}_{n,m} = \widetilde{M}_{n,m}.
\]
\end{theorem}

\begin{proof}
Let $s \in [n]^m$ be a parking function on $\widetilde{T} \in \widetilde{\mathcal{T}}_n$, and let $v \in V(\widetilde{T})$. Our goal is to extend $s$ to $s' \in [n]^n$ in a reversible manner, then apply $\psi$. We must do so in a way that is not affected by the change in edges caused by $\psi$, which suggests we avoid using edges along the path root$(\widetilde{T}) \rightarrow v$. 

To this end, drivers choose to park as follows. We recursively define $\{A_i\}_{i=1}^m$ so that $A_1 = \{s_1\}$ and in general $A_i$ are the spots that driver $i$ could park at so that the remaining drivers may successfully park, given the first $i-1$ drivers are parked. Let $B_i \subseteq A_i$ be the vertices of $A_i$ that are reachable from $s_i$ by utilizing a minimal number of edges in the path root$(\widetilde{T}) \rightarrow v$. From there, driver $i$ parks at the vertex with label $\min(B_i)$. Once driver $i$ is parked, we construct $A_{i+1}$ and continue until all drivers have parked.

Let $\{x_i\}_{i=1}^{n-m}$ be the unoccupied spaces after the drivers have parked in this manner, ordered in an increasing manner. We then define $s'$ as follows:
\[
s_i' = 
\begin{cases}
s_i &\tx{ if } i \leq m \\
x_j &\tx{ if } i = m+j
\end{cases}
\]

Then, we apply $\psi$ to $(\widetilde{T},s',v)$. Since $s'$ does not change under $\psi$ and is a parking function, $s$ is also a parking function on the resulting mapping digraph $\widetilde{M}_f$. In order to reverse, we must be able to extend $s$ to $s'$ on $\widetilde{M}_f$. Because the edges on the path between root$(\widetilde{T})$ and $v$ become the cycle edges, drivers park as defined in the first paragraph, but instead of utilizing a minimal number of path edges, they use a minimal number of cycle edges.
\end{proof}

\section{Concluding Remarks and Future Research} \label{sec:conclusion}

In this paper, we gave several equivalent characterizations of an extension of the ``drivers searching for a parking space'' description of parking functions to digraphs. Additionally, we follow the work of Lackner and Panholzer to show many of their results on trees with edges oriented towards a root still hold when the edge orientation is reversed. Furthermore, we showed that source trees never had fewer $(n,n)$-parking functions than sink trees.

We propose here some generalizations of on other notions of classical parking functions. Prime parking functions were defined by Gessel \cite{EC2} and can be understood as classical parking functions for which every edge in the path is necessary for parking. More formally,

\begin{definition}
A classical parking function $s \in [n]^n$ is \emph{prime} if, for all $2 \leq i \leq n$, we have
\[
|\{j:s_j \geq i \}| < n-i+1.
\]
\end{definition}
So let us say
\begin{definition}
A parking function $(D,s)$ is called \emph{prime} if, for every $A \subseteq [n]$ such that $R_D(A) \neq [n]$, we have
\[
|\{C_i : s_i \in R_D(A)\}| < |R_D(A)|.
\]
\end{definition}

In fact, the parking functions for which every edge is used by a driver after failing to park at her desired spot described in the proof of Theorem \ref{thm:treenumber} are prime parking functions. See \cite{KING2019} for more about prime parking functions on sink trees.

Another type of interesting parking function are the \emph{increasing} parking functions, those for which $s_i \leq s_{i+1}$ for all $i$. For the classical case, the set of increasing parking functions is counted by the famous Catalan numbers and is easily shown to be in bijection with Dyck paths of semilength $n$. We follow the terminology of \cite{BUTLER2017} and call the generalization on digraphs \emph{parking distributions}, to focus on the distribution of driver preference rather than the ordering of the sequence. Here, $f(i)$ may be understood as the number of drivers preferring vertex $i$.

\begin{definition}
Let $f: [n] \rightarrow [m]_0$ such that $\sum_{i=1}^nf(i)=m$. Then we say $(D,f)$ is a \emph{parking distribution} if for all $A \subseteq [n]$, we have
\[
\sum\limits_{i \in R_{D}(A)}f(i) \leq |R_D(A)|.
\]
\end{definition}

We present several avenues for future research:

\begin{enumerate}
\item We are interested in exact counts of $P(D,m)$ on specific digraphs or sums over families of digraphs. In particular, we do not know $\widetilde{F}_{n,m}$.
\item We are also interested in a formula more specific than that given in Theorem \ref{thm:treenumber}, describing the relationship between $P(T,n)$ and $P(\widetilde{T},n)$.
\item As noted, it is possible for some $m<n$ and $T$ to have $P(T,m) > P(\widetilde{T},m)$. We ask for a characterization of when this occurs.
\item Both the path that classical parking functions are defined on and sink trees as a family support interesting numbers of parking functions. Are there other digraphs or families of digraphs which are associated with particularly nice numbers of (prime) parking functions or parking distributions?
\item We ask how one may extend several of the statistics defined on classical parking functions, such as the number of ``lucky'' drivers, those who park in their preferred spot, or the ``total displacement'', the distance driven by all drivers. In general, the locations in which drivers park is not well-defined, but perhaps we could get around this by considering a maximum or minimum over all successful parking outcomes.
\end{enumerate}

\section*{Acknowledgments}

The authors would like to thank Professor Chun-Hung Liu for suggesting the inductive proof of Lemma \ref{lem:cycles}, which simplified our original argument.

\bibliographystyle{abbrv}
\bibliography{SourceReferences}

\begin{thebibliography}{10}

\bibitem{ARMSTRONG2016}
D.~Armstrong, N.~A. Loehr, and G.~S. Warrington.
\newblock Rational parking functions and catalan numbers.
\newblock {\em Annals of Combinatorics}, 20(1):21--58, Mar 2016.

\bibitem{BUTLER2017}
S.~Butler, R.~Graham, and C.~H. Yan.
\newblock Parking distributions on trees.
\newblock {\em European Journal of Combinatorics}, 65:168 -- 185, 2017.

\bibitem{EHRENBORG2016}
R.~Ehrenborg and A.~Happ.
\newblock Parking cars of different sizes.
\newblock {\em The American Mathematical Monthly}, 123(10):1045--1048, 2016.

\bibitem{GESSEL2006}
I.~M. Gessel and S.~Seo.
\newblock A refinement of cayley's formula for trees.
\newblock {\em Electr. J. Comb.}, 11(2), 2006.

\bibitem{KINGFPSAC2018}
W.~King and C.~H. Yan.
\newblock Parking functions on oriented trees.
\newblock {\em S{\'e}minaire Lotharingien de Combinatoire}.
\newblock To appear.

\bibitem{KING2019}
W.~King and C.~H. Yan.
\newblock Prime parking functions on rooted trees.
\newblock {\em J. of Comb. Theory Series A}.
\newblock To appear.

\bibitem{KONHEIM1966}
A.~G. Konheim and B.~Weiss.
\newblock An occupancy discipline and applications.
\newblock {\em SIAM Journal on Applied Mathematics}, 14(6):1266--1274, 1966.

\bibitem{KUNG2003}
J.~P. Kung and C.~Yan.
\newblock Gon\u{c}arov polynomials and parking functions.
\newblock {\em Journal of Combinatorial Theory, Series A}, 102(1):16 -- 37,
  2003.

\bibitem{LACKNER2016}
M.-L. Lackner and A.~Panholzer.
\newblock Parking functions for mappings.
\newblock {\em Journal of Combinatorial Theory, Series A}, 142:1 -- 28, 2016.

\bibitem{CAMERON2008}
T.~P. Peter J~Cameron, Daniel~Johannsen and P.~Schweitzer.
\newblock Counting defective parking functions.
\newblock {\em Electronic Journal of Combinatorics}, 15, 2008.

\bibitem{POSTNIKOV2004}
A.~Postnikov and B.~Shapiro.
\newblock Trees, parking functions, syzygies, and deformations of monomial
  ideals.
\newblock {\em Transactions of the American Math Society}, 142, 2004.

\bibitem{EC2}
R.~Stanley.
\newblock {\em Enumerative Combinatorics}, volume~2.
\newblock Cambridge University Press, 1999.

\bibitem{YAN2015}
C.~H. Yan.
\newblock Parking functions.
\newblock In M.~B{\'o}na, editor, {\em Handbook of Enumerative Combinatorics},
  chapter~13, pages 835--894. CRC Press, 2015.

\end{thebibliography}

 
\end{document}